\documentclass[11pt,a4paper]{amsart}

\usepackage{amssymb,amsmath,graphics,verbatim}
\usepackage{latexsym}
\usepackage{eucal}
\usepackage{a4wide}

\newtheorem{proposition}{Proposition}[section]
\newtheorem{theorem}{Theorem}[section]
\newtheorem{lemma}[theorem]{Lemma}
\newtheorem{prop}[theorem]{Proposition}

\newcommand{\mc}{\mathcal}
\newcommand{\cal}{\mathcal}
\newcommand{\rr}{\mathbb{R}}

\newcommand{\nn}{\mathbb{N}}
\newcommand{\C}{\mathbb{C}}
\newcommand{\cc}{\mathbb{C}}

\newcommand{\eps}{\varepsilon}

\newcommand{\pl}{\partial}
\newcommand{\x}{\times}

\newcommand{\til}{\widetilde}
\newcommand{\bbar}{\overline}

\newcommand{\cjd}{\rangle}
\newcommand{\cjg}{\langle}

\newcommand{\demi}{\frac{1}{2}}

\newcommand{\tra}{\textrm{Tr}}

\newcommand{\la}{\lambda}

\def\qed{\hfill$\square$\medskip}

\newcommand{\bbC}{\mathbb{C}}
\newcommand{\CI}{C^{\infty}}
\newcommand{\End}{\mathrm{End}}
\newcommand{\rest}[1]{\big\rvert_{#1}} % restriction e.g. to boundary
\newcommand{\pa}{\partial}

\newcommand\paperintro%
        {%
         }
\newcommand\paperbody%
        {%
         }
\numberwithin{equation}{section}
\numberwithin{theorem}{section}

\begin{document}

\title[Inverse Boundary Problems for Systems in Two Dimensions]{Inverse Boundary Problems for Systems in Two Dimensions}

\author{Pierre Albin}
\address{Institut de Math\'ematiques de Jussieu \\
U.M.R. 7586 CNRS \\
Universit\'e Paris 7\\
175 rue du Chevaleret\\
75013 Paris}
\email{albin@math.jussieu.fr}

\author{Colin Guillarmou}
\address{D\'epartement de Math\'ematiques et Applications\\
U.M.R. 8553 CNRS \\
Ecole Normale Sup\'erieure \\
45 rue d'Ulm, F 75230 Paris \\
cedex 05, France}
\email{cguillar@dma.ens.fr}

\author{Leo Tzou}
\address{ Department of Mathematics, \\ The University of Arizona\\ 
617 N. Santa Rita Ave. P.O. Box 210089. \\ Tucson, AZ 85721Ð0089 USA}
\email{leo.tzou@gmail.com}

\author{Gunther Uhlmann}
\address{ University of California, \\ Irvine, 340 Rowland Hall\\ Irvine, CA 92697-3875,
and Department of Mathematics\\ The University of Washington\\
C-449 Padelford Hall\\
Box 354350\\
Seattle, Washington 98195-4350
USA}
\email{gunther@math.washington.edu}

\thanks{P. A. was supported by a postdoctoral fellowship
  of the Foundation Sciences Math\'ematiques de Paris.}
\thanks{C.G. is supported by ANR Grants No. ANR-09-JCJC-0099-01 and ANR-10-BLAN 0105.} 
\thanks{L.T. was supported by NSF Grant No. DMS-0807502.}
\thanks{G.U. was supported by NSF, a Chancellor Professorship at UC Berkeley and a 
Senior Clay Award.} 

%G.U. thank G. Nakamura and J-N. Wang for helpful discussions on elasticity.}

\begin{abstract}
We prove identification of coefficients up to gauge by Cauchy data at the boundary 
for elliptic systems on oriented compact surfaces with boundary or domains of $\cc$. 
In the geometric setting, we fix a Riemann surface with boundary, and consider both a Dirac-type operator plus potential acting on sections of a Clifford bundle and a connection Laplacian plus potential
 (i.e. Schr\"odinger Laplacian with external Yang-Mills field)  acting on sections of a Hermitian bundle. 
 In either case we show that the Cauchy data determines both the connection and the potential up to a natural gauge transformation: conjugation by an endomorphism of the bundle which is the identity at the boundary. 
%This generalizes previous work of two of the authors on the inverse problem for the connection Laplacian on sections of a line bundle over a Riemann surface with boundary.
%A corollary of our main
%result is that we can determine the connection and the potential from
%Cauchy data for the Schr\"odinger equation in an external Yang-Mills
%field. 
For domains of $\cc$, we recover zeroth order terms up to gauge from Cauchy data at the boundary in 
first order elliptic systems. 
%As an application, we
%prove identification of the Lam\'e parameters from the
%Dirichlet-to-Neumann map (or Cauchy data) associated to the system of
%isotropic elasticity in a simply connected domain. A similar result is also valid for inverse
%scattering at a fixed energy. 
\end{abstract}
\maketitle

\paperintro
\section*{Introduction}
%\begin{section}{Introduction}

In this work, we show that the Cauchy data space at the boundary identifies the coefficients (up to gauge) 
of a certain type of first order and second order elliptic systems on a Riemann surface and domains of $\cc$, generalizing the results of \cite{GTgafa}. We show here that Cauchy data at the boundary of a Riemann surface determine: \\
1) the $0$-th order term (up to gauge invariance) in the operator $D+V$ where $D$ is a Dirac type operator and $V$ an endomorphism on a clifford bundle,\\ 
2) the connection $\nabla$ and the potential  $V$ (up to gauge) acting on a complex vector bundle in the Schr\"odinger 
connection Laplacian $\nabla^*\nabla +V,$\\
%3) the Lam\'e parameters in the isotropic elasticity equation in a  simply connected domain.\\   
%From a mathematical physics point of view, the Schr\"odinger connection Laplacian in 2) is the second order system studied by Schrader-Taylor in \cite{SchTay} and corresponds to Schr\"odinger equation in an external Yang-Mills
%field. 

\subsection{Connection Laplacians on surfaces, Schr\"odinger operator with external Yang-Mills field}
Let $M$ be a Riemann surface with boundary and $\pi : E\to M$ be a complex vector bundle equipped with a Hermitian structure $\langle\cdot,\cdot\rangle_E$. We denote the space of $E$-valued $k$-forms by $\Omega^k(E) = \CI(M; \Lambda^kT^*M \otimes E)$ and similarly use $\Omega^{p,q}(E)$ to denote $E$ valued forms of type $(p,q)$. Let $\nabla$ be a connection on $E$ and consider the connection Laplacian $L := \nabla^*\nabla$ where the adjoint $\nabla^*$ is taken with respect to the Hermitian inner product. For $V\in L^\infty(M,{\rm End}(E))$, we define the Cauchy data space of the operator $L+V$ by 
\begin{equation}\label{Cauchydata}
\mc{C}_{L+V}
:=\{ (u, \nabla_\nu u)|_{\pl M}\in H^{\demi}(\pl M,E)\x H^{-\demi}(\pl M,E); (L+V)u=0 ,  u\in H^1(M,E)\}
\end{equation} 
where $\nu$ is the inward normal vector field to the boundary. 

The Cauchy data space can not determine the connection $\nabla$ and the potential $V$, for there is a gauge invariance.
Indeed, it suffices to consider the conjugation of $L+V$ by a unitary section 
\begin{equation*}
	F \in \CI(M; \End(E)), \quad F^* = F^{-1}, \quad F|_{\pl M}={\rm Id}.
\end{equation*}
There is a natural lift of $F$ to a unitary endomorphism in $\CI(M;\End(T^*M\otimes E))$, still denoted $F$, 
defined by $(F\sigma)(X) := F (\sigma (X))$ for all $\sigma \in \CI(M;T^*M\otimes E)$ and $X \in TM,$ and 
it is easy to see that the Cauchy data space of $\til{L}+\til{V}=(F^{-1} \nabla F)^*(F^{-1} \nabla F)+F^{-1}VF$
is the same as the Cauchy data space of $L+V,$
\begin{equation*}
\mc{C}_{\til{L}+\til{V}}=\mc{C}_{L+V}.
\end{equation*} 

In this paper, we prove that the Cauchy data space determines $\nabla$ and $V$ up to gauge.  
Before we state the result, we use the notation $C^r(M)$ (with $r\geq 0$) for the usual $r-$H\"older space 
on $M$ and $W^{s,p}(M)$ (with $p\in[1,\infty]$, $s\in\rr$) for the Sobolev space with $s$ derivatives in $L^p(M)$, 
while $H^s(M):=W^{2,p}(M)$. 

A connection on a vector bundle $E$ is said to be in $C^r$ (or similarly $W^{s,p}$) if its (local) connection form 
is in $C^{r}(M; {\rm End}(E)\otimes T^*M)$. 
\begin{theorem}
\label{main thm}
Let $\nabla_1$ and $\nabla_2$ be two Hermitian connections on a smooth Hermitian bundle $E,$ of complex dimension  $n$ and   let $V_1,$ $V_2$ be two sections of the bundle ${\rm End}(E)$.
We assume that $\nabla_j$ have the regularity $C^r\cap W^{s,p}(M)$ with 
\begin{equation}\label{regrsp}
0 < r < s, \quad p \in (1,\infty) \,\,\,{\rm satisfy }\,\,
	r+s>1, \quad r\notin\nn, \quad sp>2n+2
	\end{equation}
and that $V_j\in W^{1,q}(M)$ with  $q>2$. Let $L_j:=\nabla_j^*\nabla_j$ and assume that 
the Cauchy data spaces agree $\mc{C}_{L_1+V_1}=\mc{C}_{L_2+V_2},$ then
there exists a unitary endomorphism 
$F \in C^1(M; \End(E)),$ satisfying $F|_{\pl M}={\rm Id},$
such that $\nabla_1 = F^{-1} \nabla_2 F$ and $V_1=F^{-1}V_2F$.
\end{theorem}
Observe that Theorem \ref{main thm} is a generalization of the scalar trivial bundle case in \cite{GTgafa} where $E = M \x \C$ and $\nabla_j = d + iX_j$ for $X_j$ a real valued 1-form. For scalar trivial bundle on domains of $\cc$, this was 
first proved (with partial data measurement) by Imanuvilov-Uhlmann-Yamamoto \cite{IUY2,IUY2'}.
Our result is new even in the case of domains in 
$\cc$ when the bundle is not a line bundle (this was known only under smallness assumption, see Li 
\cite{Li}).

We prove Theorem \ref{main thm} as  a consequence of an identifiability result for Dirac-type systems.

\subsection{Dirac systems on surfaces}

A Dirac vector bundle (also known as a Clifford vector bundle) on a Riemannian manifold $(M,g)$ 
is a complex vector bundle 
$E \longrightarrow M$ together with a Clifford multiplication map,
\begin{equation*}\label{cliffordmult}
	\gamma: \CI(M; T^*M) \longrightarrow \CI(M; \End(E)), \quad \text{ s.t. }
	\gamma(\eta)\gamma(\omega) + \gamma(\omega)\gamma(\eta) 
	= -2 g(\eta, \omega),
\end{equation*}
a Hermitian metric $\langle \cdot,\cdot \rangle_E$ and a Hermitian connection $\nabla$ satisfying
\begin{equation}\label{compatibility}
	\langle \gamma(\omega) s, t \rangle_E
	= - \langle s, \gamma(\omega) t \rangle_E, \quad
	\left[ \nabla_W, \gamma(\omega) \right] = \gamma( \nabla_W\omega )
\end{equation}
for every $\omega \in \CI(M; T^*M),$ $s,t \in \CI(M, E),$ and $W \in \CI(M; TM).$
In dimension $2$, there is a chirality operator defined by
\begin{equation*}
H:=  i\gamma(\theta^1)\gamma(\theta^2) 
\end{equation*}
where $(\theta^1, \theta^2)$ is any local orthonormal basis of $T^*M$. One easily checks that $H$ does not depend on 
the choice of  $(\theta^1, \theta^2)$ and therefore it can be defined globally on the surface. 
Since $H^2 = \mathrm{Id},$ it determines a splitting of $E,$
\begin{equation*}
	E = E^+ \oplus E^-,
\end{equation*}
the connection preserves this splitting and Clifford multiplication reverses it. 
The Dirac-type operator associated to this data is the composition
\begin{equation}
	D: H^1(M;E) \xrightarrow{\nabla}
	L^2(M; T^*M \otimes E) \xrightarrow{\gamma}
	L^2(M; E).
\label{DiracOpDef}\end{equation}
It is self-adjoint with respect to $\langle \cdot, \cdot \rangle_E,$ and odd with respect to the splitting of $E.$ 
As above, if $V\in W^{1,p}(M,{\rm End}(E))$, we can define the Cauchy data space of $D+V$ by 
\[\mc{C}_{D+V}:=\{ u|_{\pl M}\in H^{\demi}(\pl M,E); (D+V)u=0 ,  u\in H^1(M,E)\}.\]
We then prove the following theorem: 
\begin{theorem}
\label{dirac ident}
Let $(E, \langle\cdot,\cdot\rangle, \gamma, \nabla_j)$, $j = 1,2$ be two Dirac bundles on a Riemann surface $M$ with boundary. We assume the bundle and $\gamma$ are smooth, while $\nabla_j$ is $C^r\cap W^{s,p}$ with $s,p,r$ as in \eqref{regrsp}. 
 Let $V_1,V_2$ be $W^{1,q}$ sections of ${\rm End}(E)$, with $q>2$. Suppose that the Cauchy data spaces $\mc{C}_{D_1+V_1}$ and $\mc{C}_{D_2+V_2}$ coincide, then there exist  $C^1$ bundle-morphisms $\Phi,\Psi: E\to E$, preserving the splitting $E=E^+\oplus E^-$, with $\Psi=\Phi={\rm Id}$ on $\pl M$ and such that $\Phi(D_1+V_1)\Psi = D_2+V_2$.
\end{theorem}
%It turns out that there is an additional bonus of proving a general identifiability result for general Dirac-type systems. We will see that the problem of recovering Lam\'e elasticity parameters of a planar material is a consequence of this fact. 
%\lt{$C_{D+V_1}$ should be $C_{D_1 + V_1}$ and $H_j$ should be $D_j$. At the moment your clifford multiplication is the same so chirality operator is the same as well}
As we shall prove, this theorem actually follows from the particular case:
\begin{prop}\label{firstorderonM}
Let $\underline{\cc}^n:=M\x \cc^n$ and $E:=\underline{\cc}^n\oplus (\underline{\cc}^n\otimes (T^{0,1}M)^*)$. 
Consider the operators $D+V_j$, $j=1,2$, defined by 
 $D:=\begin{pmatrix}0 & \bar{\pl}^*\cr 
\bar\pl & 0\cr\end{pmatrix}$, $V_j=\begin{pmatrix}Q_j^+ & A_j\\
 B_j & Q_j^-\end{pmatrix}$, 
acting on sections of $E$, with $A,B\in W^{1,q}(M)$, $q>2$, and $Q^\pm\in C^r\cap W^{s,p}(M)$ such that $r,s,p$ satisfies
the condition \eqref{assump}.  
If the Cauchy data spaces $\mc{C}_{D+V_1}$ and $\mc{C}_{D+V_2}$ agree, then there exist $C^1$ bundle isomorphisms $F,G $ of $\underline{\bbC}^n$ such that  $F|_{\pl M}=G|_{\pl M}={\rm Id}$ and, as operators, 
\begin{equation*}%\label{conjugV_1V_2}
 D+V_2=  \begin{pmatrix}
	G & 0 \\ 0 & F^{-1}
	\end{pmatrix} (D+V_1) \begin{pmatrix}
	F & 0 \\ 0 & G^{-1}
	\end{pmatrix}.
\end{equation*}
\end{prop}

The proof of this identification result is based on previous work \cite{GTgafa} of the second and third authors, itself based on a new idea of Bughkeim \cite{Bu}, and the construction of holomorphic phases on Riemann surfaces in \cite{GTAust}. 

\subsection{Systems in domains of $\cc$}
For domains $\Omega$ of $\cc$, our proof shows the identification up to gauge of zeroth order terms 
$V$ by Cauchy data 
for any elliptic $(m+n)\x (m+n)$ systems of the form 
\[  \left(\begin{matrix} 
\bar{\pl} & 0 \\
0 & \pl 
\end{matrix}\right) \left(\begin{matrix} 
u \\
v 
\end{matrix}\right) + \left(\begin{matrix} 
A & Q^+ \\
Q^- & B 
\end{matrix}\right)\left(\begin{matrix} 
u \\
v 
\end{matrix}\right)=0
\]
where $\bar{\pl}$ acts on each component  of $\cc^m$ valued functions $u$ by $\pl_{\bar{z}}$ and 
$\pl$ acts on each component of $\cc^n$ valued functions $v$ by $\pl_{z}$, and $Q^\pm,A,B$ are matrix valued 
functions satisfying similar asumptions as in Proposition \ref{firstorderonM}. We refer to Theorem 
\ref{systeminC} below for a precise statement.\\

\subsection{State of the art in two dimensions}

Let us recall some known results about Calder\'on inverse type problem in dimension 2 (we do not 
discuss here references of higher dimensional results). 

For inverse problems on domains in $\cc$, 
Nachman \cite{Nach} proved that the Cauchy data space determines a 
$C^2$-conductivity (with a reconstruction method). Before that, Sylvester \cite{Syl}Ê showed how to reduce the problem for anisotropic conductivities to isotropic conductivities.

Brown and Uhlmann \cite{BrUh} used the $\bar{\pl}$ factorization of Beals-Coifman \cite{BeCoi} for solving the identification result for the isotropic conductivity with regularity $W^{1,p}$ for $p>2$.
The most general result in terms of regularity is for $L^\infty$-conductivity by Astala-P\"aiv\"arinta  \cite{AP} (and \cite{ALP} for anisotropic case), using quasiconformal methods. 
The identification of an $L^\infty$-potential in the Schr\"odinger operator in a domain of $\cc$  was proved recently
by Bukhgeim \cite{Bu}, after the problem had been open for more than 20 years. 
This was then extended  with only partial data measurements for domains in $\cc$
by Imanuvilov-Uhlmann-Yamamoto \cite{IUY}.
For magnetic Schr\"odinger operators, Kang-Uhlmann \cite{KaUh} showed identification of a magnetic field (up to gauge) and a potential in a simply connected domain under smallness assumption in $L^\infty$-norm of the potential. For general elliptic second operators on domains in $\cc$, and with only partial data measurements, the identification was shown by Imanuvilov-Uhlmann-Yamamoto \cite{IUY2, IUY2', IUY3}, while identification for partial data with disjoints measurements for Dirichlet and Neumann data has been proved by the same authors in \cite{IUY3}.

For what concerns elliptic systems, Novikov-Santacesaria \cite{NoSa} considered recently $\Delta+V$ acting on vector valued functions on a domain of $\cc$, with $V$ a metric potential, they show identification of $V$ from Cauchy data space. Li \cite{Li} proved identifiability up to gauge for Yang-Mills Schr\"odinger system under smallness assumption
on the coefficients.

In the geometric case, the determination of the conformal class of a metric $g$ from the Cauchy data space of $\Delta_g$ on a Riemann surface with boundary was first shown by Lassas-Uhlmann \cite{LaUh}, followed by Belishev \cite{Be} and Henkin-Michel  \cite{HeMi1} (with a reconstruction procedure). The identification 
of an isotropic conductivity on a Riemann surface was shown by Henkin-Michel \cite{HeMi} with reconstruction; it was extended by Henkin-Santacesaria \cite{HeSa} to anisotropic conductivities for surfaces embedded in $\rr^3$.  
Guillarmou-Tzou \cite{GTAust,GTduke} show identification of a $C^{1,\alpha}$-potential 
in the Schr\"odinger potential on a fixed Riemann surface (with only partial measurement), generalizing the result of \cite{IUY} to Riemann surfaces. In \cite{GTgafa}, it is proved that one can determine a connection (up to gauge) and a potential in the connection Laplacian with potential on a complex line bundle (from the full Cauchy data space).\\
%For what concerns the elasticity equation in dimension 2, this is the first global result (to our knowledge) 
%proving identification of Lam\'e parameters. 
%This was shown however for parameters close to constants by Nakamura-Uhlmann \cite{NakUhl}.\\

%We would also like to add that inverse problems for Dirac-type operators in the time dependent inverse was studied by Kurylev-Lassas in \cite{lassaskurylev}. In their work the authors established the framework for the types of boundary measurements one can obtain for Dirac-type systems and the proper formulation of an inverse problem for such systems. 

The paper is organized as follows: we first prove that a Dirac type system on a Riemann surface with boundary can be reduced to a $\bar{\pl},\bar{\pl}^*$ type system on the trivial bundle $M\x \cc^{2n}$. Then we prove Theorem \ref{dirac ident} by using the Complex Geometric Optics method (also called Faddeev exponential solutions) developed by Bughkeim \cite{Bu} for domains in $\cc$, and extended by \cite{GTgafa} to Riemann surfaces. Finally we show that the inverse problem for the connection Laplacian  of Theorem \ref{main thm} can be reduced to Theorem \ref{dirac ident}. 
%In the last section we apply the result of Theorem \ref{dirac ident} to solve the problem of recovering the Lam\'e elasticity parameters of a planar material.

%%%%%%%%%%%%%%%%%%%%%%%%%%%%%
\paperbody
\begin{section}{Dirac operators and $\bar\pl$ operators}
%%%%%%%%%%%%%%%%%%%%%%%%%%%%%

In this section, we show that any system of  Dirac  type $(D+V)u=0$ on a Dirac vector bundle can be reduced to a 
$\bbar{\pl}, \bbar{\pl}^*$ system.
We will work with holomorphic structures of low regularity following \cite{HiTay}.\\

If $\mc{B}$ is a regularity space (such as a Sobolev space or a H\"older space),
a  holomorphic structure with regularity $\mc{B}$ on a smooth complex vector bundle $E$ of rank $n$  over a Riemann surface $M$ with boundary is an atlas of local trivializations 
$h_i: E|_{U_i}\to U_i\x \cc^n$, $i\in I$, with regularity $\mc{B}$  such that the transition functions 
$h_{ij}: U_i\cap U_j\to GL(N,\cc)$ are holomorphic with respect to the holomorphic structure on 
$M$. 
A holomorphic section $s: U\to E$ is a section such that $h_i\circ s$ is holomorphic as a map from 
$U$ to $\cc^n$. 

Following Hill and Taylor \cite{HiTay}, we will usually ask that our holomorphic structures have regularity
$C^{r+1}\cap W^{s+1,p}$ where 
\begin{equation}
	0 < r < s, \quad p \in (1,\infty) \text{ satisfy }
	r+s>1, \quad r\notin\nn, \quad sp>2n+2.
\label{HiTayReg}\end{equation}

A holomorphic vector bundle over a surface with boundary is holomorphically trivial.
This is shown for instance in Forster \cite[Th 30.1 and Th 30.4]{Fo} for smooth holomorphic structures, and with small modification yields:
\begin{prop}\label{holotrivial}  
Let $E\to M$ be a holomorphic vector bundle of rank $n$ with regularity $C^r\cap W^{s,p}$
 over a compact Riemann surface $M$ with non-empty boundary, where $r,s,p$ are as in \eqref{HiTayReg}.
Then $E$ is holomorphically trivial in the sense that there exist $n$ holomorphic sections  
$f_1,\dots, f_n\in C^{r+1}\cap W^{s+1,p}(M,E)$  
such that at every point $x\in M$, $f_1(x),\dots,f_n(x)$ are linearly independent in the fiber $E_x$.
\end{prop}

Closely associated to a holomorphic structure on a complex vector bundle $E \to M$ is a Cauchy-Riemann type operator. We say that a first order differential operator
\begin{equation*}
	P: H^1(M; E) \longrightarrow L^2(M; (T^{0,1}M)^* \otimes E)
\end{equation*}
is a {\em CR operator} if it satisfies
\begin{equation*}
	P(f \xi) = f P(\xi) + (\bar\pl f) \xi, \quad
	\text{ for every } f \in \CI(M), \xi \in \CI(M, E).
\end{equation*}
A CR operator $P$ can be extended to forms and satisfies $P^2=0$ since on surfaces there are no $(0,2)$-forms.
If a Hermitian product is given on $E$, a CR operator $P$ is induced by (and induces) a unique Hermitian connection $\nabla$ in the sense that
\begin{equation*}
	P\xi = (\nabla \xi)^{0,1}
\end{equation*}
and we will denote this operator by $\bar{\pl}^\nabla.$
 We say that a holomorphic structure on $E$ is compatible with a CR operator 
$\bar{\pl}^\nabla$ if the holomorphic sections of $E$ are in $\ker \bar{\pl}^\nabla$.

Given a regularity space $\mc{B},$
we say that $\bar{\pl}^\nabla$ is a Cauchy-Riemann operator of class
$\mc{B}$  if there is a  holomorphic atlas on $M$ in which 
$\bar{\pl}^\nabla$ can be written locally as
\[\bar{\pl}^\nabla= \bar{\pl}+A\]
with $A\in \mc{B}(M,{\rm End}(E)\otimes (T^{0,1}M)^*).$

\begin{lemma}\label{equivCR}
Let $E\to M$ be a smooth complex vector bundle of rank $n$ and let $r,s,p$ be as in \eqref{HiTayReg}. 
There is a $C^{r+1}\cap W^{s+1,p}$ holomorphic structure on $E$ 
if and only if $E$ can be 
equipped with a CR operator $P$ of class 
$C^r\cap W^{s,p}.$
\end{lemma}
\begin{proof} 
Given a $C^{r+1}\cap W^{s+1,p}$ holomorphic structure on $E,$ let
\begin{equation*}
	F = (f_1, \ldots, f_n): E\to M \times \bbC^n 
\end{equation*}
where $f_1,\ldots,f_n$ are the holomorphic sections from Proposition \ref{holotrivial}.
Then $F$ is a bundle isomorphism of class $C^{1+r}\cap W^{s+1,p}.$ 
If $\bar{\pl}$ is the  standard Cauchy-Riemann operator on the trivial bundle $M\x \cc^n$ induced by the $\bar{\pl}$ operator on $\CI(M),$ then $P = F \bar{\pl} F^{-1}$ is a CR operator on $E$ of class $C^r\cap W^{s,p}$ 
compatible with the holomorphic structure on $E.$

Conversely, given $P$ a CR operator of class $C^r\cap W^{s,p},$ we can find a holomorphic structure on $E$ following the proof of Kobayashi \cite[Ch. 1, Prop 3.7]{Ko}. This proof is in the smooth category, but we can modify it slightly to have it in the range of regularity assumed above, by using a low regularity version of  the famous Newlander-Nirenberg integrability result \cite{NewNi} due to Hill-Taylor \cite{HiTay}. 
To construct the holomorphic local trivializations, Kobayashi defines an almost-complex structure $J$ on $E$ which is integrable, and  with $J$ having the regularity of $A$ when we write the Cauchy-Riemann operator as $\bar{\pl}+A$ in a smooth trivialization $M\x \cc^n$ of $E$. Consider the same $J$ as Kobayashi (in the proof of Prop 3.7 of \cite{Ko}), it is $C^r\cap W^{s,p}$ since $A\in C^r\cap W^{s,p}$ by assumption and $J$ is formally integrable by the argument of Kobayashi (which just comes from $P^2=0$). 
We can then use the main result of Hill-Taylor \cite{HiTay}  which says that an integrable complex structure 
$J$ which is $C^r\cap W^{s,p}$ on a complex manifold of dimension $n+1$ induces local holomorphic trivializations of $E$ of class $C^{1+r}\cap W^{s+1,p}$ with the $s,r,p$ satisfying the conditions above.  
\end{proof}

We now show that on a surface with boundary, any Dirac-type operator \eqref{DiracOpDef} on a holomorphic vector bundle is induced by a CR-operator.
\begin{lemma}
\label{to dbar}
Let $(E,\cjg\cdot,\cdot\cjd, \gamma, \nabla)$ be a Dirac vector bundle of complex rank $2n$ over a smooth Riemann surface  $M$  with boundary,  and assume that  $\nabla$ and $\gamma$ have 
the regularity $C^{r}\cap W^{s,p}$ with $r,s,p$ 
satisfying \eqref{HiTayReg}. Let $D$ be the associated Dirac operator.
Then there exists a complex subbundle $E_0$ of $E$ of complex rank $n$ and 
a  bundle isomorphism
\begin{equation*}
	B: E \longrightarrow 
	 E_0 \oplus (E_0 \otimes (T^{(0,1)}M)^*)
\end{equation*}
such that 
\begin{equation*}
	B D B^{-1} = \sqrt{2}( \bar\pl^{\nabla} + (\bar\pl^{\nabla})^* ) 
\end{equation*}
where $\bar\pl^{\nabla}$ is the CR operator associated with the Hermitian connection $\nabla$ on $E.$
\end{lemma}

\begin{proof}
Endow $E$ with the holomorphic structure induced by the Hermitian connection $\nabla$ through its $\bar{\pl}^\nabla$, this is a $C^{1+r}\cap W^{1+s,p}$ holomorphic structure. 
Both $T^{1,0}M$ and $E$ are  trivial holomorphic bundles by Proposition \ref{holotrivial}.

Choose (pointwise orthonormal) sections 
$Z \in \CI(M; T^{1,0}M),$ $\bar Z \in \CI(M; T^{0,1}M)$
and let $S,$ $T$ be real sections in $\CI(M; TM)$ such that $T = JS,$ where $J\in {\rm End}( TM)$ is the complex structure on $TM,$ and 
\begin{equation*}
	\bar Z = \tfrac12 (S + i T)
\end{equation*}
(note that $|S| = |T| = \sqrt2$).
After extending $\gamma$ to be $\bbC$-linear and defining $S^*,T^*$ the dual basis to $S,T$ and $Z^*:=S^*+iT^*$, 
$\bar{Z}^*:=S^*-iT^*$ we can write the Dirac-type operator as
\begin{equation*}
	D
	= \gamma(S^*)\nabla_{S} + \gamma(T^*)\nabla_{T}
	= \gamma(Z^*) \nabla_{Z} + \gamma(\bar Z^*)\nabla_{\bar Z}
\end{equation*}

Notice that
\begin{equation}
	\gamma(Z^*)^2 = \gamma(\bar Z^*)^2 = 0 \quad \text{ and } \quad
	\gamma(Z^*)\gamma(\bar Z^*) + \gamma(\bar Z^*)\gamma(Z^*)
	= -2,
\label{AlgRel}\end{equation}
and hence the image of $\gamma(Z^*)$ is equal to its null space, and  $\gamma(\bar{Z}^*)$ establishes an isomorphism between the image of $\gamma(Z^*)$ and its orthogonal complement (recall that 
$\gamma(\bar{Z}^*)^*=-\gamma(Z^*)$ by \eqref{compatibility}).
That is, we have
\begin{equation*}
	E = E_0 \oplus \gamma(\bar{Z}^*) E_0, \text{ with } E_0 = \gamma(Z^*) E.
\end{equation*}
The bundle $E_0$ is a complex subbundle of rank $n$ of $E$, and is trivial on $M$.
We can then define 
\begin{align*}
	B: E= E_0 \oplus   \gamma( \bar{Z}^*) E_0&\longrightarrow  E_0\oplus ( (T^{0,1}M)^* \otimes E_0), \\
	  \quad v + \gamma(\bar{Z}^*)w &\longmapsto w + \sqrt{2} v \bar Z^*
\end{align*}
and from \eqref{AlgRel} we see that
\begin{equation*}
	B(\gamma(\bar Z^*)\cdot) = \sqrt2 \mathfrak{e}(\bar Z^*) B(\cdot), \quad
	B(\gamma(Z^*) \cdot) = -\sqrt{2} \mathfrak{i}(\bar Z^*) B(\cdot)
\end{equation*}
where $\mathfrak e$ denotes the exterior product and $\mathfrak i$ denotes the interior product.
Since we also have \cite[Remark C.1.3]{McDuff-Salamon}
$(\bar\pa^{\nabla})^* = -i * \pa^{\nabla} = -\mathfrak{i}(\bar Z^*)\nabla_{Z},$
this allows us to identify
\begin{equation*}
	B( D \cdot)
	= \sqrt{2} \left(
	\bar Z^* \wedge \nabla_{\bar Z} - \mathfrak{i}(\bar Z^*) \nabla_{Z}
	\right) B(\cdot)
	= \sqrt 2 \left( \bar\pa^{\nabla} + (\bar \pa^{\nabla})^* \right) B(\cdot).
\end{equation*}
as required.
\end{proof}

It will be useful to recall the Lichnerowicz formula of Bochner-Kodaira for the Dirac-type operator associated to the CR-operator, $\sqrt{2}( \bar\pl^{\nabla} + (\bar\pl^{\nabla})^* ),$ namely \cite[Proposition 3.71]{Berline-Getzler-Vergne}
\begin{equation*}
	\bar\pl^{\nabla}(\bar\pl^{\nabla})^* + (\bar\pl^{\nabla})^*\bar\pl^{\nabla}
	= \nabla^*\nabla
	+ \sum_{j,k} \mathfrak{e}(d\bar z^j) \mathfrak{i}(dz^k) F^{E_0\otimes K^*}
	(\pa_{z_k}, \pa_{\bar z_j})
	= \nabla^*\nabla + \Omega_{\Lambda^*( T^{0,1}M)^* \otimes E_0}.
\end{equation*}
where $F^{E_0\otimes K^*}$ is the twisting curvature of the connection, which here reduces to 
the curvature  $2$-form, $\Omega_{\Lambda^*( T^{0,1}M)^* \otimes E_0}.$
The latter equality holds as we are working on a surface, which also means that
\begin{equation*}
	\bar\pl^{\nabla}(\bar\pl^{\nabla})^* + (\bar\pl^{\nabla})^*\bar\pl^{\nabla}
	= \begin{cases}
	(\bar\pl^{\nabla})^*\bar\pl^{\nabla} & \text{ on } E_0 \\
	\bar\pl^{\nabla}(\bar\pl^{\nabla})^* & \text{ on } E_0 \otimes T^*_{0,1}M
	\end{cases}
\end{equation*}

Since $E \to M$ is holomorphically trivial, Proposition \ref{holotrivial} gives us a bundle isomorphism
$F = (f_1, \ldots, f_{2n}): E\to M \times \bbC^{2n} $ such that $F \bar\pl^{\nabla} = \bar\pl F.$ Taking adjoints,
\begin{equation*}
	(\bar \pl^{\nabla})^* F^* = F^* (\bar \pl)^*,
\end{equation*}
and hence the Bochner-Kodaira formula yields
\begin{equation}
	\nabla^*\nabla
	= - (\bar\pl^{\nabla})^*\bar\pl^{\nabla} - \Omega
	= - F^* (\bar \pl)^* (F^*)^{-1} 
	F^{-1} \bar\pl F
	- \Omega
\label{FactForm}\end{equation}
over $E_0$ and a similar formula over $(T^{0,1}M)^* \otimes E_0.$
This factorization formula will be useful for our considerations below.

\end{section}
\section{Identification for Dirac systems}

In this section we will prove Theorem \ref{dirac ident}: the Cauchy data of a Dirac operator plus potential on a surface with boundary identifies the operator up to a unitary endomorphism equal to the identity at the boundary.

In the previous section, the bundle $E_0$ is a complex subbundle of rank $n$ on $M$, it is then trivial and can be identified with $\underline{\cc}^n :=M\x \cc^n$. Moreover, by Lemma \ref{to dbar}, we may assume that the Dirac bundle is  a $\bar{\pl}$ system on a trivial $\cc^{2n}$  bundle, so in this section we will study operators of the form
\begin{equation}\label{D+V}
	D+V 
	:= \begin{pmatrix}
	0 & \bar\pl^* \\ \bar\pl & 0
	\end{pmatrix}+ \begin{pmatrix}
	Q^+ & A'^* \\ A & Q^-
	\end{pmatrix} 
\end{equation}
acting as a bounded operator from $H^1(M,E)$ to $L^2(M,E)$  where $E:= \underline{\cc}^n\oplus (\underline{\cc}^n\otimes (T^{0,1}M)^*)$, 
\begin{equation}\label{assump}
\begin{gathered}
	A, A' \in C^r\cap W^{s,p}(M,{\rm End}(\cc^n)\otimes (T^{0,1}M)^*), 
	\text{ with } r,s,p \text{ as in \eqref{HiTayReg},} \\
	 \textrm{ and }ÊQ^\pm \in W^{1,q}(M) , \, q>2.
\end{gathered}
\end{equation}
Notice that by Sobolev embedding, $Q^{\pm} \in L^{\infty}(M)$.

First consider the case where $A = A'=0,$ so that we want to determine $Q^{\pm}$ from the Cauchy data.
The proof of \cite[Lemma 7.4]{GTgafa} applies {\em verbatim} to a matrix valued potential and shows that
\begin{equation}
	Q^{\pm}\rest{\pl M} \text{ are determined by } \mathcal C_{D+V};
\label{BdyDetQ}\end{equation}
in \S\ref{sec:IdPot} we will show that in fact $\mathcal C_{D+V}$ determines 
$Q^{\pm}$ on all of $M.$
We do this, following the arguments in \cite{Bu} and \cite{GTgafa}, by 
showing that for a dense set of points $z \in M$ one can find a solution
of `complex geometric optics' (or CGO) type that determines the potential
at that point. These solutions are constructed in \S\ref{sec:CGO} and then in 
\S\ref{sec:RedDiag} we show that we can reduce the general case to the case $A = A'=0.$

\subsection{Solutions from complex geometric optics}\label{sec:CGO}
In this section, we use the method of \cite{GTgafa}, based on the work of Bughkeim \cite{Bu} to construct elements in the null space of $D+V$ of CGO-type (where $V$ is as in \eqref{assump} with $A=A'=0$).
These solutions will have Morse holomorphic or anti-holomorphic phases, so we start by recalling the following Proposition, proved in \cite[Prop 2.1]{GTAust}: 
\begin{prop}\label{denseset}
On a Riemann surface with boundary, there exists a dense set of points $p$ in $M$ such that there exists a holomorphic function $\Phi$ which is Morse, i.e. $d\Phi$ has only zeros of order $1$, and with a critical point at $p$.
\end{prop}

Let $\Phi$ be a holomorphic Morse function and $h>0$ a small parameter.
For certain choices of  
\begin{equation} \label{abCond}
	a \in \CI(M; \underline{\bbC}^n), \quad \bar\pl a =0, \quad
	b \in \CI(M; T^*_{0,1}M \otimes \underline{\bbC}^n), \quad \bar\pl^*b=0
\end{equation}
we will show that one can find $r_h,$ $s_h$ with small $L^p$ norms as $h\to 0$
so that 
\[ U_h=\begin{pmatrix}  e^{\Phi/h}(a+r_h)  \\ e^{\overline\Phi/h}(b+s_h)  \end{pmatrix}\]
is in the null space of $D+V.$

Since $\bar\pl \Phi = 0$ and $\bar \pl^*\overline\Phi =0,$ solving $(D+V)U_h=0$ for $r_h,$ $s_h$ is equivalent to
\begin{equation*}
	 (D+V_{\psi})
	\begin{pmatrix}  a+r_h  \\ b+s_h \end{pmatrix}
	= {0 \choose 0}
\end{equation*}
where $\psi:={\rm Im}(\Phi)$ and 
\begin{equation*}
	V_\psi:= 
	\begin{pmatrix} e^{-\overline\Phi/h} & 0 \\ 0 & e^{-\Phi/h} \end{pmatrix} V
	\begin{pmatrix} e^{\Phi/h} & 0 \\ 0 & e^{\overline\Phi/h} \end{pmatrix}
	= \left(\begin{matrix}
		e^{\frac{2i\psi}{h}}Q^+ & 0\\
		0 & e^{-\frac{2i\psi}{h}}Q^-
	\end{matrix}\right),
\end{equation*}
and hence to
\begin{equation}\label{eqtosolve}
	(D+V_\psi)\begin{pmatrix}  r_h  \\ s_h \end{pmatrix}=-V_\psi{a \choose b}.
\end{equation}

Next we make use of a right inverse for $D$ as constructed in Proposition 2.1 and Lemma 2.1 of the paper 
\cite{GTgafa}. 
\begin{prop}\label{inverse}
There exists an operator $D^{-1}:L^q(M,E)\to W^{1,q}(M,E)$, bounded for all $q\in (1,\infty)$, such that  
$DD^{-1}={\rm Id}$.
\end{prop}
\begin{proof}
The proof in \cite{GTgafa} is written for the case of a line bundle, i.e. $n=1$,  but since here the bundle is the trivial $\cc^n$ bundle, the operator $D$ splits as a direct sum of operators on complex line bundles and the proof is then contained in \cite{GTgafa}. The operator $D^{-1}$ is of the form 
\begin{equation}\label{DInv}
	D^{-1}=\begin{pmatrix}
	0 & \mc{R}\bar\pl^{-1}\mc{E} \\  \mc{R}(\bar\pl^{*})^{-1}\mc{E} & 0 
	\end{pmatrix}
\end{equation}
for some operators $\bar\pl^{-1}$ and $(\bar\pl^{*})^{-1}$ inverting on the right the operators $\bar\pl$ 
and $(\bar\pl^{*})^{-1}$ on an open manifold $\til{M}$ which contains $M$, $\mc{E}$ is an extension operator extending $W^{k,q}(M)$ sections on $M$ to compactly supported sections in $W^{k,q}(\til{M})$  
(for $k\in\{0,1\},$ $q\in [1,\infty]$)
and $\mc{R}$ is a restriction operator, restricting sections in $L^{q}(\til{M})$ to sections in $L^{q}(M)$, see Section 2 of \cite{GTgafa} for more details (where the notation for $M\subset \til{M}$ was $M_0\subset M$).  
\end{proof}

Let $\psi$ be a real valued smooth Morse function on $\til{M}$ and let 
%\lt{I think here you want $\psi$ to be a function on $\tilde M$}
\begin{equation*}
	\bar{\pl}^{-1}_\psi:=\mc{R}\bar{\pl}^{-1}e^{-\frac{2i\psi}{h}}\mc{E}, \qquad
	(\bar{\pl}_\psi^*)^{-1}:=\mc{R}(\bar{\pl}^*)^{-1}e^{2i\psi/h}\mc{E}   
\end{equation*}
%\lt{and $e^{-2i\psi/h}$ should be on the LEFT of the extension operator}
with notation as in \eqref{DInv}. From \cite[Lemma 2.2 and 2.3]{GTgafa}, we have following estimates 
\begin{lemma}\label{estimate1}
For any $q>2$, there exists $\eps>0$ and $C>0$ such that for all 
$\omega\in W^{1,q}(M,\underline{\bbC}^n),$ 
$\omega'\in W^{1,q}(M,\underline{\bbC}^n\otimes (T^{0,1}M)^*),$
and $h>0$ small
\begin{equation*}
	||(\bar{\pl}_\psi^*)^{-1}\omega||_{L^2}\leq Ch^{\demi+\eps}||\omega||_{W^{1,q}}, \quad
	||\bar{\pl}^{-1}_\psi \omega'||_{L^2}\leq Ch^{\demi+\eps}||\omega'||_{W^{1,q}}.
\end{equation*}
Also, for
$v\in W^{1,q}(M,\underline{\bbC}^n)$ and $v'\in W^{1,q}(M,\underline{\bbC}^n\otimes (T^{0,1}M)^*)$ satisfying $v|_{\pl M}=0$ and $v'|_{\pl M}=0,$ we have 
\begin{equation}\label{estadj}
\begin{gathered}
	||\mc{E}^* (\bar{\pl}^{-1})^*\mc{R}^*(e^{-2i\psi/h}v)||_{L^2}
	\leq Ch^{\demi+\eps}||v||_{W^{1,q}},\\
	||\mc{E}^* ((\bar{\pl}^*)^{-1})^*\mc{R}^*(e^{-2i\psi/h}v')||_{L^2}
	\leq Ch^{\demi+\eps}||v'||_{W^{1,q}}
\end{gathered}
\end{equation} 
\end{lemma}

Notice that, with
\begin{equation*}
	D_{\psi}^{-1} = \begin{pmatrix}
	0 & \bar\pl_{\psi}^{-1} \\
	(\bar\pl_{\psi}^*)^{-1} & 0 
	\end{pmatrix},
\end{equation*}
we have $D^{-1}V_{\psi} = D_{\psi}^{-1}V.$
Hence applying $D^{-1}$ to both sides of \eqref{eqtosolve} yields
\begin{equation*}
	(\mathrm{Id}+D_{\psi}^{-1}V)\begin{pmatrix}  r_h  \\ s_h \end{pmatrix}
	=-D_{\psi}^{-1}V{a \choose b}.
\end{equation*}
and hence
\begin{equation}\label{s-r equation}
	\left\{\begin{array}{ll}
	r_h + \bar{\pl}_\psi^{-1}(Q^-s_h) = -\bar{\pl}^{-1}_\psi (Q^-b)\\
	s_h +(\bar{\pl}_\psi^*)^{-1}(Q^+r_h) = -(\bar{\pl}_\psi^*)^{-1}(Q^+a)
	\end{array}\right.
\end{equation}

Next we specialize to $a=0$ and find that $r_h$ satisfies
\begin{eqnarray}\label{solve for r}
	(I - S_h) r_h = -\bar{\pl}_\psi^{-1}(Q^- b) \,\,\, 
	\textrm{ with }S_h:= \bar{\pl}_\psi^{-1}Q^-\bar{\pl}_\psi^{*-1}Q^+ .
\end{eqnarray}
where $Q^+,Q^-$ are viewed as multiplication operators. 
Similarly to Lemma 3.1 in \cite{GTgafa}, we have the following consequence of Lemma \ref{estimate1}:
\begin{lemma}\label{normestim}
Let $q>2$ and assume that $Q^+\in L^\infty(M,{\rm End}({\bbC}^n))$ and $Q^-\in W^{1,q}(M,{\rm End}({\bbC}^n))$, 
then $S_h$ is bounded on $L^r(M)$ for any $1<r\leq q$ and satisfies $||S_h||_{L^r\to L^r}=\mc{O}(h^{1/r})$ if 
$r>2$ and $||S_h||_{L^2\to L^2}=\mc{O}(h^{1/2-\eps})$ for any $0<\eps<1/2$ small.
\end{lemma}

Using Lemma \ref{normestim} equation \eqref{solve for r} can be solved, for small enough $h,$ through a Neumann series
\begin{equation}\label{r neumann}
	r_h := -\sum\limits_{j = 0}^\infty S_h^j\bar{\pl}_\psi^{-1}Q^- b
\end{equation}
which defines an element of $L^q(M)$ for any $q\geq 2.$
Substituting this expression for $r_h$ into equation \eqref{s-r equation} when $a=0$, we get that
\begin{equation}\label{shnorm}
	s_h = -(\bar{\pl}_\psi^*)^{-1}Q^+r_h.
\end{equation}
But using Lemma \ref{normestim} and \ref{estimate1} for $s_h$ and $r_h$, we deduce that if 
$Q^+\in L^\infty(M,{\rm End}({\bbC}^n))$ and $Q^-\in W^{1,q}(M,{\rm End}({\bbC}^n))$ for some $q>2$, then there exists $\eps >0$ such that
\[\|s_h\|_{L^2(M_0)} + \|r_h\|_{L^{2}(M_0)}=\mc{O}(h^{\demi+\eps}).\]

Similarly, if one assumes that 
$Q^-\in L^\infty(M,{\rm End}({\bbC}^n))$ and $Q^+\in W^{1,q}(M,{\rm End}({\bbC}^n))$ for some $q>2$, then
one can solve the system \eqref{s-r equation}, with $b=0,$ for any holomorphic $a.$
In summary:
\begin{proposition} \label{first order CGO a = 0}
Let $\Phi = \phi + i\psi$ be a Morse holomorphic function on $M$, and 
$b$ an anti-holomorphic section of $\underline{\bbC}^n\otimes \Lambda^{0,1}(M)$. 
If $Q^+\in L^\infty (M,{\rm End}({\bbC}^n))$ and 
$Q^-\in W^{1,q}(M,{\rm End}({\bbC}^n))$ for some $q>2$, 
there exist solutions to $(D+V)F_h= 0$ on $M$ of the form
\begin{equation}\label{Fh} 
	F_h = \begin{pmatrix}
		e^{\Phi/h} r_h \\
		e^{\bar{\Phi}/h} (b + s_h)
	\end{pmatrix}
\end{equation}
where $\|s_h\|_{L^2} + \|r_h\|_{L^{2}}=\mc{O}(h^{\demi+\eps})$ for some $\eps>0.$

If instead
$Q^-\in L^\infty(M,{\rm End}({\bbC}^n))$ and $Q^+\in W^{1,q}(M,{\rm End}({\bbC}^n))$ for some $q>2$, 
then for any 
holomorphic section of $\underline{\bbC}^n,$ $a,$ 
there exist solutions to $(D+V)G_h= 0$ on $M$ of the form
\begin{equation} \label{Gh}
	G_h = \begin{pmatrix}
		e^{\Phi/h} (a+r_h) \\
		e^{\bar{\Phi}/h} s_h
	\end{pmatrix}
\end{equation}
where $\|s_h\|_{L^2} + \|r_h\|_{L^{2}}=\mc{O}(h^{\demi+\eps})$ for some $\eps>0.$ 
\end{proposition}

Note that for arbitrary $(a,b)$ satisfying \eqref{abCond}, one can just add the solutions for $(a,0)$ and $(0,b).$

\subsection{Identifying the potential}\label{sec:IdPot}

As explained above, the CGO-type elements in the null space of $D+V$ with $V$ a diagonal potential, suffice to identify the potential on all of $M.$

\begin{proposition}\label{idpotential}
Let 
$D=\begin{pmatrix}0 & \bar{\pl}^* \\ \bar\pl & 0 \end{pmatrix}$, let $V_1,$ $V_2$ be two sections of ${\rm End}(E)$ with $A=A'=0$ when written in the form \eqref{D+V},Êand 
satisfying the regularity assumption \eqref{assump}. 
If the Cauchy data spaces $\mc{C}_{D+V_1}$ and $\mc{C}_{D+V_2}$ agree, then $V_1 = V_2$ on all of $M.$
\end{proposition}

\begin{proof} 
Given the construction of CGO-type solutions above, the proof now follows that of \cite[Thm. 3.3]{GTgafa} (and hence also the argument of Bughkeim \cite{Bu}), we repeat the main arguments for the reader's convenience.
Let $\Phi$ be a Morse holomorphic function with a critical point at $z_0$. The existence of such a function for a dense set of points $z_0$ of $M_0$ is insured by Proposition \ref{denseset}.
For any $b_1,$ $b_2$ anti-holomorphic sections of $\underline{\bbC}^n\otimes (T^{1,0}M)^*,$ we can find 
\[
F^1_h:=\begin{pmatrix}{e^{\Phi/h} r^1_h}\cr{e^{\bar\Phi/h}(b_1 + s^1_h)}\end{pmatrix},  \quad 
F^2_h:=\begin{pmatrix}{e^{-\Phi/h} r^2_h}\cr{e^{-\bar\Phi/h}(b_2 + s^2_h)}\end{pmatrix}
\]
solutions of   $(D+V_1)F^1_h=0$  and $(D+V_2^*)F^2_h=0$ as in \eqref{Fh}, Êwhere $r^j_h,$ $s^j_h$ are constructed in Proposition \ref{first order CGO a = 0}.

Since $\mc{C}_{D+V_1}=\mc{C}_{D+V_2}$, there exists an $F_h$ satisfying
\begin{equation*}
	(D+V_2)F_h=0 \quad \text{ and } \quad
	i_{\pl M}^*F_h=i_{\pl M}^*F^1_h.
\end{equation*}
In particular, $(D+V_2)(F^1_h-F_h)=(V_2-V_1)F^1_h$ and 
$i_{\pl M}^*(F^1_h-F_h)=0$. We use Green's formula and the vanishing of $F^1_h-F_h$ on the boundary
to get 
\begin{equation}\label{integid}
0=\int_{M}\cjg (D+V_2)(F^1_h-F_h),F_h^2\cjd =\int_{M}\cjg (V_2-V_1)F^1_h,F^2_h\cjd.
\end{equation}
where $\cjg\cdot,\cdot\cjd$ denotes the Hermitian scalar product on $E$.
If we denote $Q_2^\pm - Q_1^\pm$ by $Q^{\pm},$ we can rewrite this as 
\begin{equation}\label{integralid}
	0=
	\int_{M}e^{-2i\frac{\psi}{h}}
	\Big( \cjg Q^{-}b_1,b_2\cjd
	+ \cjg Q^{-}b_1,s_h^2\cjd
	+ \cjg Q^{-}s_h^1,b_2\cjd 
	+ \cjg Q^{-}s_h^1,s_h^2\cjd \Big)
	+ e^{2i\frac{\psi}{h}}\cjg Q^+r^1_h,r^2_h\cjd.
\end{equation}
By Proposition \ref{first order CGO a = 0}, we always have
\begin{equation}\label{firstremain}
	\int_{M}e^{-2i\frac{\psi}{h}}	 \cjg Q^{-}s_h^1,s_h^2\cjd +
	e^{2i\frac\psi h}\cjg Q^+r^1_h,r^2_h\cjd=\mc{O}(h^{1+\eps})
\end{equation}
for some $\eps>0$.
Next we choose 
$b_1=\theta e_j$ and $b_2=\theta e_k$, where $\theta$ is an anti-holomorphic $1$-form which vanishes 
at all critical points of $\Phi$ in $M$ except at the critical point $z_0\in M$ of $\Phi$ (and $e_j,$ $e_k$ denote vectors in the canonical basis of $\bbC^n$). 
The existence of $\theta$ is guaranteed by the Riemann-Roch theorem (see Lemma 4.1 in \cite{GTAust}). 
We observe by using stationary phase that, as $h\to 0$,  
\begin{equation}\label{statphase}
	\int_{M}e^{-2i\psi/h}\cjg Q^-b_1,b_2\cjd 
	= C_{z_0} h e^{-2i\psi(z_0)/h}\cjg Q^-(z_0)e_j,e_k\cjd |\theta(z_0)|^2+o(h)
\end{equation}
for some constant $C_{z_0}\not=0$. To show this, one can decompose the integral, using a smooth cutoff function, near $z_0$ and far from $z_0$. The part localized near $z_0$ is simply obtained by stationary phase while 
the other part is $o(h)$, as can be seen by integrating by parts once (using $\pl_z e^{i\psi/h}=ie^{i\psi/h}(\pl_z\psi)/h $) to gain an $h$ factor and then applying Riemann-Lebesgue to show that the remaining 
oscillating integral goes to $0$ as $h\to 0.$ There are no boundary terms in the integration by parts since, by \eqref{BdyDetQ}, $Q^-|_{\pl M}=0.$ 

Let us now consider the term with $\cjg Q^-b_1,s_h^2\cjd$ in \eqref{integralid}: using \eqref{shnorm} we can write this as
\begin{equation*}
	\int_{M}\cjg e^{-2i\psi/h}  Q^-b_1,s_h^2\cjd
	=\int_{M} e^{-2i\psi/h}\cjg  \mc{E}^*(\bar{\pl}^{*-1})^*\mc{R}^*e^{-2i\psi/h}Q^-b_1 ,  
	Q_2^+r^2_h \cjd.
\end{equation*}
Since $Q^-|_{\pl M}=0$, we may use \eqref{estadj} to deduce that  
$||\mc{E}^*(\bar{\pl}^{*-1})^*\mc{R}^*e^{2i\psi/h}Q^-b_1||_{L^2}=\mc{O}(h^{1/2+\eps})$ and thus combining with Proposition 
\ref{first order CGO a = 0}, we find that  
\[\int_{M}\cjg e^{2i\psi/h}  Q^-b_1,s_h^2\cjd=\mc{O}(h^{1+\eps}).\]
the same argument shows that the term involving $\cjg Q^-s_h^1,b\cjd$ in \eqref{integralid} is 
$\mc{O}(h^{1+\eps})$. 
These last two estimates combined with \eqref{statphase} and \eqref{firstremain} imply that $Q^-_1(z_0)=
Q^-_2(z_0)$ by letting $h\to 0$.
The same proof using the complex geometric optics solution $G_h$ of Proposition \ref{first order CGO a = 0} gives $Q^+_1(z_0)=\til{Q}^+_2(z_0)$.
\end{proof}

\subsection{Reduction to the diagonal case} \label{sec:RedDiag}

Proposition \ref{idpotential} is the result we wanted when the potential is diagonal. We now show that one can always reduce to the diagonal case, and thereby establish the full result.

We first rewrite the operator $D+V$ as follows 
\[ 
D+V=  \begin{pmatrix}
	0 & (\bar\pl+A')^* \\ \bar\pl+A & 0
	\end{pmatrix}
	+\begin{pmatrix} Q^+ & 0 \\ 0 & Q^- \end{pmatrix}.
\]
The operators $\bar{\pl}_A:=\bar{\pl}+A$ and $\bar{\pl}_{A'}:=\bar{\pl}+A'$ are CR operators on the trivial bundle $\underline{\bbC}^n=M\x \cc^n$ over $M$, so by Lemma \ref{equivCR} they induce holomorphic structures  on $\underline{\bbC}^n$ and by Proposition \ref{holotrivial} there are holomorphic trivializations
$F_A,$ $F_{A'}\in C^{1+r}\cap W^{1+s,p}(M,{\rm End}(\underline{\bbC}^n))$ such that 
\begin{equation}\label{HoloTrivDef}
	F_A^{-1}\bar{\pl}F_A= \bar{\pl}_A , \quad F_{A'}^{-1}\bar{\pl}F_{A'}=\bar{\pl}_{A'}. 
\end{equation}
From the factorization
\begin{equation}\label{factorization}
 D+V=  \begin{pmatrix}
	F_{A'}^* & 0 \\ 0 & F_{A}^{-1}
	 \end{pmatrix} \left [
	\begin{pmatrix}
	0 & \bar\pl^* \\ \bar\pl & 0
	\end{pmatrix}+ \begin{pmatrix}
	(F_{A'}^{*})^{-1}Q^+F_A^{-1} & 0 \\ 0 & F_{A}Q^-F_{A'}^*
	\end{pmatrix}\right]  \begin{pmatrix}
	F_{A} & 0 \\ 0 & (F_{A'}^{*})^{-1}
	 \end{pmatrix}.
\end{equation}
we see that the Cauchy data space of $D+V$ is determined by the Cauchy data space of $D+\til{V}$ and the boundary values of $F_{A},$ $F_{A'},$  where 
\begin{equation*}
	\til{V}:= \begin{pmatrix}
	(F_{A'}^{*})^{-1}Q^+F_A^{-1} & 0 \\ 0 & F_{A}Q^-F_{A'}^*
	\end{pmatrix}.
\end{equation*}
(Under assumption \eqref{assump}, one has $\til{V}\in W^{1,q}(M,{\rm End}(\cc^n))$ for $q>2$.)
We will show that, when the Cauchy data space of two Dirac-type operators coincide, it is possible to choose these isomorphisms consistently at the boundary.

\begin{prop}\label{boundary holo}
Let $A_i,$ $Q_i^\pm$ satisfy the regularity assumption \eqref{assump}, for $i=1,2$. 
If
 $\begin{pmatrix}Q^+_1& (\bar\pl + A_1')^*\cr \bar\pl+A_1&Q^-_1\cr\end{pmatrix}$ and $\begin{pmatrix}Q^+_2& (\bar\pl + A_2')^*\cr \bar\pl+A_2&Q^-_2\cr\end{pmatrix}$ 
have the same Cauchy data at $\pl M,$ 
then there exist bundle isomorphisms $\til F_{A_1},$ $\til F_{A_1'}$ and $\til F_{A_2},$ $\til F_{A_2'}$ as in \eqref{HoloTrivDef} such that $\til F_{A_1}=\til F_{A_2}$ on $\pl M$ and $\til F_{A_1'}=\til F_{A_2'}$ on $\pl M.$
\end{prop}

\begin{proof}
Let $F_{A_1},$ $F_{A_1'}$ and $F_{A_2},$ $F_{A_2'}$ be as in \eqref{HoloTrivDef}, we will show that one can modify $F_{A_1'}$ and $F_{A_2'}$ to satisfy the requirements of the proposition.

Our main tools will be the CGO-type solutions from \S\ref{sec:CGO} and the following orthogonality condition for a function on $\pl M$ to extend holomorphically into the interior. The condition was derived in \cite{GTgafa} and essentially is a basic computation using the Hodge decomposition. We refer the reader to Lemma 4.1 of \cite{GTgafa} for a detailed proof.
\begin{lemma}
\label{orthogonality condition}
A complex valued function $f\in H^{1/2}(\partial M)$ is the restriction of a holomorphic function if and only if
\[\int_{\partial M} f i_{\partial M}^*\eta = 0\]
for all 1-forms $\eta\in C^{\infty}(M ; (T^{1,0}M)^*)$ satisfying $\bar\partial\eta = 0$.
\end{lemma}
Notice that, since $\pl: H^{1}(M)\to L^2(M;(T^{1,0}M)^*)$ is surjective, the forms $\eta$  in the lemma can be replaced by $\pl \theta$ with $\theta$ a harmonic function.

Let $\Phi$ be a holomorphic Morse function and let $a$ and $b$ satisfy \eqref{abCond}, applying Proposition \ref{first order CGO a = 0} to the system $D+\til{V_1},$ 
we can find $U_h^1$ of the form
\[ U_h^1=
\begin{pmatrix}  
e^{\Phi/h}(a+r_h^1)  \\ e^{\overline\Phi/h}(s_h^1)  \end{pmatrix}\]
such that $(D+ \til{V_1})U_h^1 =0.$
From \eqref{factorization} it follows that
\begin{equation*}
	\til{U}_h^1 = 
	\begin{pmatrix}  e^{\Phi/h}F_{A_1}^{-1}(a+ r_h^1)  \\ 
	e^{\overline\Phi/h}F_{A_1'}^*(s_h^1)  \end{pmatrix}
\end{equation*}
satisfies $(D+V_1)\til U_h^1 = 0.$

Similarly, using the holomorphic Morse function $-\Phi,$ and $b$ satisfying \eqref{abCond}, 
we can find a solution to 
\begin{equation*}
	\begin{pmatrix}(Q^+_2)^*& (\bar\pl + A_2)^*\cr \bar\pl+A_2'&(Q^-_2)^*\cr\end{pmatrix}\til U_h^2 = 0,
	\text{ of the form }
	\til U_h^2=\begin{pmatrix}  e^{-\Phi/h}F_{A_2'}^{-1}(r_h^2)  \\ 
	e^{-\overline\Phi/h}F_{A_2}^*(b+s_h^2)  \end{pmatrix}.
\end{equation*}

Note that $\til U_h^2$ is in the null space of $(D+V_2)^*.$
Using the equality of the Cauchy data for $D+V_1$ and $D+V_2,$ we can find an element $V_h$ in the null
space of $D+V_2$ whose boundary data agrees with that of $\til{U}_h^1.$
Hence we obtain no boundary terms when we apply Green's formula to see
\begin{equation*}
	\int_M \langle (D+V_2)\til U_h^1, \til U_h^2 \rangle
	= \int_M \langle (D+V_2)(\til U_h^1 - V_h), \til U_h^2 \rangle = 0.
\end{equation*}
Applying the remainder estimates from Proposition \ref{first order CGO a = 0} we find that,
 as $h\to 0,$
\begin{equation*}
\begin{gathered}
	0 = \int_M \langle (D+V_{2}) \til U_h^1, \til U_h^2 \rangle
	= \int_M \langle (V_{2} - V_1) \til U_h^1, \til U_h^2 \rangle
	= \int_M \left\langle 
		\begin{pmatrix} Q^+_2-Q^+_1 & A_2'^*-A_1'^* \\ A_2 - A_1 & Q^-_2 - Q^-_1 \end{pmatrix}
		\til U_h^1, \til U_h^2 \right\rangle \\
	= \int_M \langle (A_2-A_1) F_{A_1}^{-1}(a+r_h^1), F_{A_2}^*(b+s_h^2) \rangle 
	%+\langle (A_2'^*-A_1'^*) F_{A_1'}^*(b+s_h^1), F_{A_2'}^{-1}(a+r_h^2) \rangle 
	+ o(1) \\
	= \int_M \langle F_{A_2}(A_2-A_1) F_{A_1}^{-1}a, b \rangle 
	%+ \int_M \langle b, F_{A_1'}(A_2'-A_1')F_{A_2'}^{-1}a \rangle 
	+ o(1)
\end{gathered}
\end{equation*}
and hence 
\begin{equation*}
	\int_M \langle F_{A_2}(A_2-A_1) F_{A_1}^{-1}a, b \rangle 
%	+ \langle b, F_{A_1'}(A_2'-A_1')F_{A_2'}^{-1}a \rangle
	= 0
\end{equation*}
for all $a$ and $b$ as above.
Using the relations $\bar\pl F_{A_j} =  F_{A_j} A_j,$ 
%$\bar\pl F_{A_j'} =  F_{A_j'} A_j'$ 
this integral becomes 
\[
	 0 
	= \int_M \langle \bar\pl (F_{A_2}F_{A_1}^{-1}a), b \rangle 
	= \int_{\pl M} i_{\pl M}^* \langle F_{A_2}F_{A_1}^{-1} a,  b \rangle.
\]

We are free to choose the  holomorphic section $a$ and the antiholomorphic $1$-form $b$.  Denoting $(e_1,\dots, e_n)$ the canonical (holomorphic) basis of $\underline{\bbC}^n$,  we 
choose $a=e_k$ and $b=\pl (\theta e_j)$
where $\theta$ is a harmonic function $\bar{\pl}\pl\theta=0$.
 Then if we denote by $(F_{A_2} F_{A_1}^{-1})_{j,k}$ the $(j,k)$ component of the matrix $i_{\pl M}^* F_{A_2}F_{A_1}^{-1}$, we have, by Lemma \ref{orthogonality condition}, that each component of the endomorphism extends holomorphically into $M$ and we see that $i_{\pl M}^* F_{A_2}F_{A_1}^{-1}  $ admits a holomorphic extension $F$.
 
This function $F$ is invertible. Indeed, switching the indices $1$ and $2,$ we find a holomorphic extension function with boundary value $i_{\pl M}^* F_{A_1}F_{A_2}^{-1}.$ The composition of these functions is holomorphic and equal to the identity on $\pl M,$ hence on all of $M.$

Notice that the endomorphisms $\til F_{A_1} = F_{A_1},$ $\til F_{A_2} = F^{-1}F_{A_2}$ satisfy the requirements of Proposition \ref{boundary holo}, since 
\begin{equation*}
	\til F_{A_2}|_{\pl M} = (F^{-1} F_{A_2})|_{\pl M} = F_{A_1}|_{\pl M} = \til F_{A_1}|_{\pl M}.
\end{equation*}
Using similar arguments, we can find a holomorphic extension $F'$ of $F_{A_1'}F_{A_2'}^{-1}$ from the boundary of $M,$ and then $\til F_{A_1'} = F_{A_1'},$ $\til F_{A_2'} = (F')^{-1}F_{A_2'}$ will satisfy the requirements of Proposition \ref{boundary holo}, thus completing the proof.
\end{proof} 
To summarize, by Proposition \ref{boundary holo} we are able to conjugate the operator $D+ V_j$ into the operator $D + \til V_j$ where $\til V_j$ has block structure with only diagonal entries. Furthermore, since $\til F_{A_1} = \til F_{A_2}$ and  $\til F_{A_1'}  = \til F_{A_2'}$ along the boundary, we have that $\cal{C}_{D + \til V_1}= \cal{C}_{D + \til V_2}$. We can now apply Proposition \ref{idpotential} to $D + \til V_j$ and prove
\begin{theorem}\label{iduptogauge}
Let $V_1,V_2$ be two sections  of ${\rm End}(E)$ satisfying the regularity assumption \eqref{assump}, 
and $D:=\begin{pmatrix}0 & \bar{\pl}^*\cr 
\bar\pl & 0\cr\end{pmatrix}$. 
If the Cauchy data spaces $\mc{C}_{D+V_1}$ and $\mc{C}_{D+V_2}$ agree, then there exist $\mathcal C^1$ bundle isomorphisms $F,G $ of $\underline{\bbC}^n$ such that  $F|_{\pl M}=G|_{\pl M}={\rm Id}$ and, as operators, 
\begin{equation*}%\label{conjugV_1V_2}
 D+V_2=  \begin{pmatrix}
	G & 0 \\ 0 & F^{-1}
	\end{pmatrix} (D+V_1) \begin{pmatrix}
	F & 0 \\ 0 & G^{-1}
	\end{pmatrix}.
\end{equation*}
\end{theorem}

The proof of Theorem \ref{dirac ident} is an immediate consequence of Theorem \ref{iduptogauge} 
and Lemma \ref{to dbar}.

%%%%%%%%%%%%%%%%%%%%%%%%%%
\begin{section}{The second order case}
%%%%%%%%%%%%%%%%%%%%%%%%%%

In this section, we use the result for Dirac-type systems to establish Theorem \ref{main thm}: The 
Cauchy data of a connection Laplacian plus potential on a surface with boundary determines the
connection and the potential up to a gauge transformation equal to the identity at the boundary.

The connection between the second order system and the Dirac-type system is through the Bochner-Kodaria formula \eqref{FactForm}.
Indeed, this formula shows that 
\begin{equation*}
	u \in \CI(M; E) \text{ satisfies } (\nabla^*\nabla + W)u = 0 
\end{equation*}
is equivalent to
\begin{equation*}
	(u, \bar\pl^{\nabla}u) \in \CI(M; E \oplus (T^*_{0,1}\otimes E)) \text{ satisfies } 
	\begin{pmatrix}
	\Omega + W & (\bar\pl^{\nabla})^* \\ \bar\pl^{\nabla} & -\mathrm{Id}
	\end{pmatrix}
	\begin{pmatrix} u \\ \bar\pl^{\nabla} u \end{pmatrix}
	= \begin{pmatrix} 0 \\ 0 \end{pmatrix}.
\end{equation*}
Using the following proposition, we can readily apply the results above if we define $A$ by $\bar\pl^{\nabla} = \bar\pl +A,$ and let
\begin{equation*}
	D+V = 
	\begin{pmatrix}
	\Omega + W & (\bar\pl + A)^* \\ \bar\pl +A & -\mathrm{Id}
	\end{pmatrix}.
\end{equation*}

\begin{proposition}
The Cauchy data space of $\nabla^*\nabla + W$ determines the Cauchy data space of $D+V.$
\end{proposition}

\begin{proof}
Replacing $E$ with a trivial $\bbC^n$ bundle, we can write the connection as $\nabla = d +X$ for
some one-form $X.$ 
The boundary determination result of \cite[Prop. 4.1]{GTgafa} extends to show that 
the Cauchy data of $\nabla^*\nabla +W$ determines both $X|_{\pl M}$ and $W|_{\pl M}.$

Suppose that $L_1 = \nabla_1^*\nabla_1 + W_1$ and $L_2 = \nabla_2^*\nabla_2 + W_2$ have the same Cauchy data, and define $D + V_1$ and $D + V_2$ as above. 
Let ${ u \choose v}$ satisfy $(D+ V_1){ u \choose v} = {0 \choose 0}$ so that $u$ is in the null space of $L_1$ and $v = (\bar\pl + A_1)u.$
By assumption there is an element $w$ in the null space of $\nabla_2^*\nabla_2 + W_2$ such that
\begin{equation*}
	w|_{\pl M} = u|_{\pl M}, \quad \nabla_2(\nu) w|_{\pl M} = \nabla_1(\nu) u|_{\pl M}
\end{equation*}
where $\nu$ represents a normal vector to the boundary.
We will be done if we show that the element ${w \choose (\bar\pl + A_2)w }$ which is in the null space of $(D + V_2)$ has the same boundary values as ${u \choose v}.$
However we already know that $\nabla_2 w|_{\pl M} = \nabla_1 u|_{\pl M}$ since the connection one-form restricted to the boundary is determined by the Cauchy space of $L_j$ and the normal derivatives coincide by assumption, and hence
\begin{equation*}
	(\bar\pl + A_2)w|_{\pl M} = (\nabla_2 w)^{1,0}|_{\pl M}
	= (\nabla_1 u)^{1,0}|_{\pl M} = v|_{\pl M}
\end{equation*}
as required.
\end{proof}

Thus if $\mathcal C_{L_1} = \mathcal C_{L_2},$ we know from Theorem \ref{dirac ident} that  
there exist bundle isomorphisms $F$ and $G,$ equal to the identity at the boundary, such that
\begin{equation}\label{resultat}
	\begin{pmatrix}
	\Omega_1 + W_1 & A_1^* \\ A_1 & -\mathrm{Id}
	\end{pmatrix}
	 =  \begin{pmatrix}
	G & 0 \\ 0 & F^{-1}
	\end{pmatrix} 
	\begin{pmatrix}
	\Omega_2 + W_2 & A_2^* \\ A_2 & -\mathrm{Id}
	\end{pmatrix} \begin{pmatrix}
	F & 0 \\ 0 & G^{-1}
	\end{pmatrix} +
	\begin{pmatrix}
	0 & G\bar{\pl}^*G^{-1} \\ F^{-1}\bar{\pl}F & 0
	\end{pmatrix}.
\end{equation}
We underline the fact that the adjoints $A_j^*$ are adjoints of $A_j$ 
considered as maps from $E$
to $E\otimes (T^{(0,1)}M)^*$, where $E$ is equipped with Hermitian products 
$\cjg\cdot,\cdot\cjd_E$ and $E\otimes (T^{(0,1)}M)^*$ is equipped with the Hermitian product 
\[ \cjg u_1\otimes v_1, u_2\otimes v_2\cjd := \frac{1}{2i} \cjg u_1,u_2\cjd_E  *(v_1\wedge \bbar{v_2}).\]
where $*$ is the Hodge star operator.  The matrix valued form $A_j$ can be written 
$\mc{A}_j\otimes u_j$ for some $u_j\in \Lambda^{0,1}(M)$ and $\mc{A}_j\in {\rm End}(E)$, we define the adjoint of $A_j$ as an element of ${\rm End}(E)\otimes (T^{0,1}M)^*$ to be  
$A_j^{*_E}:= \mc{A}_j^{*}\otimes \bbar{u_j}\in  {\rm End}(E)\otimes (T^{1,0}M)^*$ where $\mc{A}_j^*$ is simply the adjoint 
of $\mc{A}_j$ with respect to $\cjg \cdot,\cdot\cjd_E$. Thus, one has $A_j^*=i* A_j^{*_E}\wedge $. 
Denoting $\nabla_j=d+X_j$, the fact that $\nabla_j$ is a Hermitian connection implies that 
$X_j=A_j-A_j^{*_E}$.
Now, \eqref{resultat} implies that 
\begin{equation}\label{impliesthat}
F=G^{-1}, \,\,\,\,\, A_1=F^{-1}A_2F+F^{-1}\bar{\pl}F,\,\,\,\,\,
A_1^*=GA_2^*G^{-1}+G\bar{\pl}^*G^{-1}.
\end{equation} 
Therefore, using also that $\bar{\pl}^*=-i*\pl $, we deduce  
\[F^{-1}X_2F=F^{-1}(A_2-A_2^{*_E})F=A_1-A_1^{*_E}-F^{-1}\bar{\pl} F-F^{-1}\pl F=X_1-F^{-1}dF.\]
This shows that $F^{-1}\Omega_2F=\Omega_1$ since $\Omega_j$ is the curvature of $\nabla_j$, but using 
again \eqref{resultat} one also get $F^{-1}W_2F=W_1$.
By \eqref{impliesthat}, we also obtain  
\[\bar{\pl}F^{-1}+(A_1F^{-1}-F^{-1}A_2)=0, \,\,\,\,  \bar{\pl}F^{*}+(A_1F^*-F^*A_2)=0\]
with $F^{-1}|_{\pl M}=F^*|_{\pl M}={\rm Id}$, and thus by uniqueness of the solution of the elliptic boundary value problem 
\[\bar{\pl} H+A_1H-HA_2=0, \quad H|_{\pl M}={\rm Id}\] 
we deduce $F^{-1}=F^*$. 
We thus have proved 
\begin{theorem}
If the two operators $L_1,L_2$ have same Cauchy data space, then there exists a  unitary bundle isomorphism
$F:E\to E$  such that $F\nabla_1F^{-1}=\nabla_2$, $FW_1F^{-1}=W_2$  and $F|_{\pl M}={\rm Id}$.
\end{theorem}

\end{section}

\begin{section}{Systems in domains of $\cc$}

In our subsequent applications we will need to consider $\pl$-type systems where the leading symbols are not self-adjoint operators. We consider in this section such first order systems in a domain $\Omega\subset \cc$. Let $D$ be the 
operator acting on $H^1(\Omega, \cc^m\oplus \cc^n)$, by
\[ D= \left(\begin{matrix}
\bar{\pl} & 0\\
0 & \pl
\end{matrix}\right),\,\, \textrm{ with }\,Ê \bar{\pl}\left(\begin{matrix}u_1\\ \dots\\ u_m \end{matrix}\right)=\left(\begin{matrix}\pl_{\bar{z}}u_1\\ \dots\\ \pl_{\bar{z}}u_m \end{matrix}\right),\,\,  \pl \left(\begin{matrix}v_1\\ \dots\\ v_n \end{matrix}\right)=\left(\begin{matrix}\pl_{z}v_1\\ \dots\\ \pl_{z}v_n \end{matrix}\right)
\]
where $\pl_z:=\pl_{x}-i\pl_y$ and $\pl_{\bar{z}}:=\pl_x+i\pl_y$ in the $z=x+iy$ complex coordinate on $\Omega$.
We consider a potential $V\in {\rm End}(\cc^m\oplus\cc^n)$ satisfying $V=\left(\begin{matrix}
A & Q^+\\
Q^- & B
\end{matrix}\right)$ with
\begin{equation}\label{regulV}
\begin{gathered}
Q^+\in W^{1,q}(\Omega; M_{n\x m}(\cc)),\,\, 
Q^+\in W^{1,q}(\Omega; M_{m\x n}(\cc)) ,\,\, q>2\\
A\in C^{r}\cap W^{s,p}(\Omega; M_{m\x m}(\cc)) ,\,\, B\in C^{r'}\cap W^{s',p'}(\Omega; M_{n\x n}(\cc)),\\
\textrm{ with }0 < r < s, \quad p \in (1,\infty) \,\,\,{\rm satisfy }\,\,
	r+s>1, \quad r\notin\nn, \quad sp>2m+2\\
	\textrm{ and }0 < r' < s', \quad p \in (1,\infty) \,\,\,{\rm satisfy }\,\,
	r'+s'>1, \quad r'\notin\nn, \quad s'p'>2n+2
\end{gathered}
\end{equation}
where $M_{m\x n}(\cc)$ denote the set of complex valued $m\x n$ matrices. 
We define the Cauchy data of the $D +V$ system by
\[{\cal C}_{D+ V} :=\left\{   \left(\begin{matrix}u|_{\pl\Omega}\\ v|_{\pl\Omega}\end{matrix}\right);  
\left(\begin{matrix}u\\ v\end{matrix}\right) \in H^1(\Omega, \cc^m\oplus \cc^n) \mid (D + V)  \left(\begin{matrix}u\\ v\end{matrix}\right) = 0  \right\}\]

The proofs of Proposition \ref{boundary holo} and Proposition \ref{idpotential} easily extend to cover this situation:
\begin{theorem}\label{systeminC}
Let $V_j =\left(\begin{matrix}
A_j & Q^+_j\\
Q^-_j & B_j
\end{matrix}\right)$ be matrix valued potentials satisfying regularity conditions in (\ref{regulV}). If ${\cal C}_{D + V_1}={\cal C}_{D + V_2}$, then there exists invertible matrices $F_j \in C^1(\Omega, {\rm End}(\C^m))$ 
and $G_j \in C^1(\Omega,{\rm End}(\C^n))$ such that $F_1 = F_2$, $G_1 = G_2$ on $\partial \Omega$ and
\[\bar\partial F_j = F_j A_j,\ \ \ \ \ \partial G_j = G_j B_j.\] 
Furthermore, $Q^+_1 = F Q^+_2 G^{-1}$ and $Q^-_1 = G Q^-_2 F^{-1}$ where
\[F := F_1^{-1}F_2,\ \ \ \ G = G_1^{-1} G_2.\] 
\end{theorem}
%In the following section we will use this result to prove uniqueness of coefficients for the planar elasticity equation.

\end{section}%%%%%%%%%%%%%%%


\begin{thebibliography}{99}

%\bibitem{AkNaSt} M. Akamatsu,  G. Nakamura, S. Steinberg, \emph{Identification 
%of Lam\'e coefficients from boundary observations}, 
%Inverse Problems, \textbf{7} (1991), 335-354.

\bibitem{AP} K. Astala, L. P\"aiv\"arinta, \emph{Calder\'on's inverse conductivity problem in the plane}, 
Ann. of Math. (2) \textbf{163}  (2006),  no. 1, 265--299.

\bibitem{ALP} K. Astala, M. Lassas, L. P\"aiv\"arinta, \emph{Calder\'on's inverse problem for anisotropic conductivity in the plane},  Comm. Partial Differential Equations  \textbf{30}  (2005),  no. 1-3, 207--224.

\bibitem{BeCoi} R. Beals and R. Coifman, \emph{The spectral problem for Davey-Stewarson and the
Ishimory hierarchies}, Nonlinear evolution equations: Integrability and
spectral methods, pages 15--23. Manchester University Press, 1988. 

\bibitem{Be}  M.I. Belishev, \emph{The Calderon problem for two dimensional manifolds by the BC- 
method}, SIAM J. Math. Anal. \textbf{35}, no 1, (2003), 172-182.

\bibitem{Berline-Getzler-Vergne} N. Berline, E. Getzler, M. Vergne \emph{Heat kernels and Dirac operators.} Corrected reprint of the 1992 original. Grundlehren Text Editions. Springer-Verlag, Berlin, 2004. x+363 pp.

\bibitem{BrUh} R. Brown, G. Uhlmann, \emph{Uniqueness in the inverse 
conductivity problem with less regular conductivities}, Comm. PDE \textbf{22} 
(1997), 1009-1027.

\bibitem{Bu} A.L. Bukhgeim, \emph{Recovering a potential from Cauchy data in the two-dimensional case}. 
 J. Inverse Ill-Posed Probl.  \textbf{16}  (2008),  no. 1, 19--33. 

%\bibitem{Es} G. Eskin, \emph{
%Global uniqueness in the inverse scattering problem for the Schršdinger operator with external Yang-Mills potentials.}%Comm. Math. Phys. \textbf{222} (2001), no. 3, 503-531.

\bibitem{Fo} O. Forster, \emph{Lectures on Riemann surfaces}, GTM 81, Springer.

\bibitem{GTAust} C. Guillarmou, L. Tzou, \emph{Calder\'on inverse problem for Schr\"odinger operator on Riemann surfaces}.  
Proceedings of the Centre for Mathematics and its Applications, Vol \textbf{44} (2010) - 
proceedings of the AMSI-ANU Workshop on Spectral Theory and Harmonic Analysis. 

\bibitem{GTduke} C. Guillarmou, L. Tzou, \emph{Calder\'on inverse Problem with partial data on Riemann Surfaces}, 
 Duke Math. J. \textbf{158} (2011), no. 1, 83--120. 

\bibitem{GTgafa} C. Guillarmou, L. Tzou, \emph{Identification of a connection from Cauchy data space on a Riemann surface with boundary}, GAFA \textbf{21}, no. 2, 393-418.

\bibitem{HeMi1}G. Henkin, V. Michel, \emph{On the explicit reconstruction of a Riemann surface from its Dirichlet-Neumann operator}, GAFA  \textbf{17} (2007), 116-155.

\bibitem{HeMi} G. Henkin, V. Michel, \emph{Inverse conductivity problem on Riemann surfaces.} 
 J. Geom. Anal.  \textbf{18}  (2008),  no. 4, 1033--1052.


\bibitem{HeNo} G. Henkin, R.G. Novikov, \emph{On the reconstruction of conductivity 
of bordered two-dimensional surface in $\rr^3$  from electrical currents 
measurements on its boundary}, arXiv:1003.4897. To appear in J. Geom. Anal.


\bibitem{HeSa} G. Henkin, M Santacesaria,\emph{Gel'fand-Calder\'on's inverse problem for anisotropic conductivities on bordered surfaces in $\rr^3$},  Inverse Problems \textbf{26} (2010), no. 9, 095011, 18 pp.

\bibitem{HiTay} C.D. Hill, M. Taylor \emph{Integrability of rough almost complex structures}
Journal of Geometric Analysis \textbf{13} (2003), No 1, 163-172.

%\bibitem{Ike} Masaru, Ikehata, \emph{A relationship between two Dirichlet to Neumann maps in anisotropic elastic plate theory.}, J. Inverse Ill-Posed Probl. \textbf{4} (1996), no. 3, 233Ð243.

\bibitem{IUY} O.Y. Imanuvilov, G. Uhlmann, M. Yamamoto, 
\emph{Global uniqueness from partial Cauchy data in two dimensions}, J. Amer. Math. Soc. \textbf{23}  (2010), 655-691. 

\bibitem{IUY2} O. Imanuvilov, G. Uhlmann,  M. Yamamoto, \emph{Partial Cauchy data for general second order operators in two dimensions.}, arXiv:1010.5791.

\bibitem{IUY2'} O. Imanuvilov, G. Uhlmann and M. Yamamoto,\emph{On determination of second order operators from partial Cauchy data}, Proceedings National Academy of Sciences, \textbf{108} (2011), 467-472.

\bibitem{IUY3} O. Imanuvilov, G. Uhlmann,  M. Yamamoto,  \emph{ Inverse boundary value problem by measuring Dirichlet data and Neumann data on disjoint sets}, preprint. 

\bibitem{KaUh} H. Kang, G. Uhlmann, \emph{Inverse problem for the Pauli Hamiltonian in two dimensions}, J. Fourier Anal. Appl. \textbf{10} (2004), no. 2, 201-215

\bibitem{Ko} S. Kobayashi, \emph{Differential geometry of complex vector bundles.} Publications of the Mathematical Society of Japan, 15. Kan™ Memorial Lectures, 5. Princeton University Press, Princeton, NJ; Iwanami Shoten, Tokyo, 1987. xii+305 pp. 

%\bibitem{McDSa} D. Mc Duff, D. Salamon, \emph{$J$-Holomorphic curves and symplectic topology}, AMS Colloquium publications Vol 52.

%\bibitem{lassaskurylev} M. Lassas, Y. Kurylev, \emph {Inverse Problems and Index Formulae for Dirac Operators.} Advances in Mathematics 221 (2009), 170-216

\bibitem{LaUh} M. Lassas, G. Uhlmann, \emph{On determining a Riemannian manifold from the Dirichlet-to-Neumann map}. 
Ann. Sci. \'Ecole Norm. Sup. (4) \textbf{34} (2001), no. 5, 771-787.

\bibitem{Li} X. Li, \emph{Inverse scattering problem for the Schroedinger operator 
with external Yang-Mills potentials in two dimensions at fixed energy}, 
Comm. Partial Differential Equations \textbf{30} (2005), no. 4-6, 451-482.

\bibitem{McDuff-Salamon} D. McDuff, D. Salamon, \emph{ {J}-holomorphic curves and symplectic topology.} American Mathematical Society Colloquium Publications, 52. American Mathematical Society, Providence, RI, 2004. xii+669 pp.

\bibitem{Nach} A. Nachman, \emph{Global uniqueness for a two dimensional inverse boundary value 
problem}, Annals of Math \textbf{143} (1996), 71Ð96.

%\bibitem{NakUhl} G. Nakamura, G. Uhlmann, \emph{Identification of Lam\'e Parameters by Boundary Measurements}, 
%Amer. J. Math., \textbf{115} (1993), No. 5, 1161--1187

%\bibitem{NU2} G. Nakamura and G. Uhlmann, \emph{Inverse elastic scattering at a fixed energy}, Journal of Inverse and Ill Posed Problems, \textbf{7} (1999), 283--288.

%\bibitem{NakUhl} G. Nakamura, G. Uhlmann, \emph{Global uniqueness for an inverse problem arising in elasticity}, Invent. Math. \textbf{118} (1994), 457-474.

\bibitem{NewNi} A. Newlander, L. Nirenberg, \emph{Complex coordinates in almost complex manifolds}, 
Ann. of Math., \textbf{65} (1957), 391--404.

\bibitem{NoSa} R.G. Novikov, M. Santacesaria \emph{Global uniqueness and reconstruction for the multi-channel Gel'fand-Calder\'on inverse problem in two dimensions}. ArXiv 1012.4667. 


\bibitem{SchTay} R. Schrader , M. Taylor \emph{Small h asymptotics for quantum partition functions associated to particles in external Yang-Mills potentials}, Comm. Math. Phys. \textbf{92} (1984), 555--594.

\bibitem{Syl} J. Sylvester, \emph{ An anisotropic inverse boundary value problem} 
Commun. Pure Appl. Math. \textbf{43} (1990) 201Ð32 

\end{thebibliography}
\end{document}